\documentclass[12pt,reqno]{amsart} %\documentclass[10pt]{jsarticle}
\usepackage{ascmac}
\usepackage{amsmath, amssymb, amsthm, amscd}
\usepackage[arrow, matrix]{xy}
\usepackage{enumerate}
\usepackage{type1cm}
\usepackage{color}
\usepackage{wrapfig}
\usepackage[pdftex]{graphicx}
\usepackage{fullpage}

\usepackage{comment}
\usepackage{url}
\usepackage[nameinlink, capitalise]{cleveref}

\newtheorem{thm}{Theorem}[section]
\newtheorem{lem}[thm]{Lemma}
\newtheorem{cor}[thm]{Corollary}
\newtheorem{prop}[thm]{Proposition}  
\newtheorem{axiom}[thm]{Axiom}

\theoremstyle{remark}
\newtheorem{ack}{Acknowledgments\!\!}

\theoremstyle{definition}
\newtheorem{dfn}[thm]{Definition}
\newtheorem{rem}[thm]{Remark} 

\newtheorem{eg}[thm]{Examples}       
\newtheorem{conj}[thm]{Conjecture}

\newtheorem{def/prop}[thm]{Definition/Proposition}

\newcommand{\ora}[1]{\overrightarrow{#1}}

\newcommand{\pmx}[1]{\begin{pmatrix}#1\end{pmatrix}}
\newcommand{\spmx}[1]{{\small \pmx{#1}}}

\usepackage[pagewise, mathlines]{lineno} %\linenumbers 

\numberwithin{equation}{section}
\usepackage{etoolbox}          %% <- for \cspreto, \csappto and \patchcmd

\newcommand*\linenomathpatch[1]{%
  \cspreto{#1}{\linenomath}%
  \cspreto{#1*}{\linenomath}%
  \csappto{end#1}{\endlinenomath}%
  \csappto{end#1*}{\endlinenomath}%
}
\newcommand*\linenomathpatchAMS[1]{%
  \cspreto{#1}{\linenomathAMS}%
  \cspreto{#1*}{\linenomathAMS}%
  \csappto{end#1}{\endlinenomath}%
  \csappto{end#1*}{\endlinenomath}%
}
\expandafter\ifx\linenomath\linenomathWithnumbers 
 \let\linenomathAMS\linenomathWithnumbers
\patchcmd\linenomathAMS{\advance\postdisplaypenalty\linenopenalty}{}{}{}
\else
  \let\linenomathAMS\linenomathNonumbers
\fi

\linenomathpatch{equation}
\linenomathpatchAMS{gather}
\linenomathpatchAMS{multline}
\linenomathpatchAMS{align}
\linenomathpatchAMS{alignat}
\linenomathpatchAMS{flalign}
\makeatletter
\patchcmd{\mmeasure@}{\measuring@true}{
  \measuring@true
  \ifnum-\linenopenaltypar>\interdisplaylinepenalty
    \advance\interdisplaylinepenalty-\linenopenalty
  \fi
  }{}{}
\makeatother

\newenvironment{nouppercase}{%
  \renewcommand{\uppercasenonmath}[1]{}}{}

\title[On Beloch's curve]{%\vspace{-2cm} %折り紙と3次方程式の解法\\〜折ることで写った点の軌跡について〜\\ 
On Beloch's curve that appears when solving real cubics with origami} %,\\ a pr\'ecis in Japanese)}
\author{Manami Niijima} %新島愛美 
\email{niijima.manami23@gmail.com}
\address{JVCKENWOOD Corporation}%株式会社JVCケンウッド 
\makeatletter
\@namedef{subjclassname@2020}{
\textup{2020} Mathematics Subject Classification}
\makeatother 



\begin{document}

\begin{nouppercase}
\maketitle
\end{nouppercase}

\begin{abstract}
The Justin--Huzita--Hatori Axiom 6 of origami related to so-called neusis constructions assures the solution of real cubic equations Beloch showed 1936. 
We investigate a certain real cubic curve $F(x,y)=0$, say, \emph{Beloch's curve} that appears in the algorithm and prove that 
its shape is determined by the sign of the Hessian $\mathcal{H}_F=-4(4p+q^2)$ at its uniquely existing singular point $P(p,q)$. 
This viewpoint would shed new light on the relationship between Axioms 5 and 6. 
\end{abstract}

\tableofcontents 

\section{Introduction}

Origami construction is generally defined by 7 Axioms (re-)founded by Justin--Huzita--Hatori in 1980's (cf.~\cite[p.40--45]{Justin1986}, \cite{LangRJ2010}).  
Among other things, Axiom 6 allows so-called neusis constructions, which is impossible in straightedge and compass constructions. 

\begin{axiom}[Justin--Huzita--Hatori Axioms 6]
Given two points $A$ and $B$ and two lines $l_1$ and $l_2$ on a plane, there exists a fold that places $A$ onto $l_1$ and $B$ onto $l_2$ simultaneously. 
In other words, it is possible to construct a certain fold line $l$ such that 
the reflections of points $A$ and $B$ in $l$ 
are placed on the straight lines $l_1$ and $l_2$, respectively.
\end{axiom} 

This Axiom 6 enables us to construct all solutions of any given real cubic equation by using a perfect piece of origami. 
The following method was initially shown by Beloch in 1936 \cite{Beloch1936}, 
based on Lill's enjoyable idea \cite{Lill1867} (see also \cite{Kato1999}, \cite{Hull2011AMM}, \cite{NakaiOno2015}). 

\begin{thm} \label{thm.construction}
Given segments with length $a, b, c \in \mathbb{R}$, every solution of the cubic formula 
\[x^3-ax^2-bx+c=0\ \cdots\ (\ast)\] may be constructed. More precisely,  

{\rm (1)} 
If we make a fold that places {\rm (I)} $A(-1, 0)$ onto $x=1$ and {\rm (II)} $P(b, a+c)$ onto $y=a-c$, 
then the $y$-intercept $r$ of the fold line $l$ is a solution of $(\ast)$. 

{\rm (2)} For any solution $x=r$ of $(\ast)$, a fold that places $A(-1,0)$ onto $A'(1,2r)$ satisfies the conditions {\rm (I)} and {\rm (II)} in {\rm (1)}. 
\end{thm} 

\begin{figure}[ht]
\centering
\includegraphics[width = 5cm]{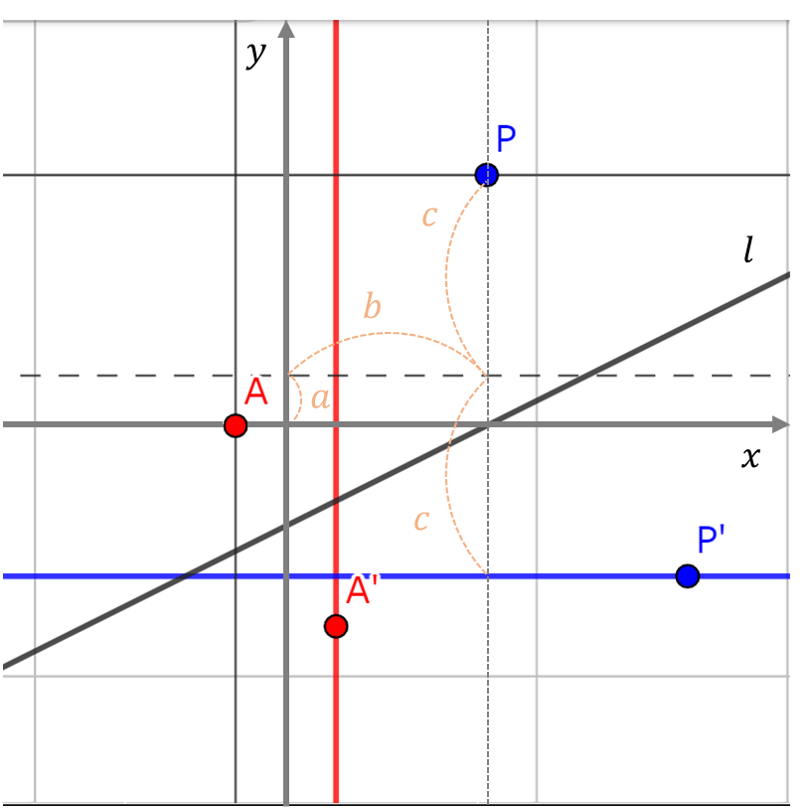}
\caption{Solving cubics by origami}
\label{drawing}
\end{figure}

In other words, \emph{there is a bijective correspondence between all fold lines satisfying the conditions {\rm (I) and (II)} and all real solutions of $(\ast)$.}  
In addition, by considering all fold lines satisfying (I) and parametrizing them by $r$, we may find all solutions. 

In this paper, we consider all fold lines satisfying the condition (I), and investigate the orbit of the points $P'(r)$ $(r\in \mathbb{R})$, each which is the reflection of $P(b,a+c)$ in a fold line. We set $(p, q)=(b, a+c)$. 
Our main results may be summarized into the following theorem.

\begin{thm} \label{thm.main} 
The union $\mathcal{F}$ of the orbit of the points $P'$ and the point $P(p, q)$ is 
a real cubic curve 
\[F(x,y):=2(q-y)^2-(q+y)(q-y)(p-x)-(p-x)^2(p+x)=0,\]
which we call \emph{Beloch's curve}.

The point $P$ is its uniquely existing singular point, 
and the Hessian at $P$ is given by $\mathcal{H}_\mathcal{F}=-4(4p+q^2)$. 
We have the following equivalence on the shape of $\mathcal{F}$ and the parabola $\mathcal{G}: 4x+y^2=0$. 

\begin{tabular}{cl}
$\bullet$ $P$ is on the left side of $\mathcal{G}$ &$\iff$ The orbit of $P'$'s does not pass through $P$\\
&$\iff$ $P$ is an isolated point of $\mathcal{F}$.
\end{tabular}

\begin{tabular}{cl}
$\bullet$ $P$ is on $\mathcal{G}$ & $\iff$ The orbit of $P'$'s passes through $P$ just once\\
&$\iff$ $P$ is a cusp of $\mathcal{F}$.
\end{tabular} 

\begin{tabular}{cl}
$\bullet$ $P$ is on the right side of $\mathcal{G}$ &$\iff$ The orbit of $P'$'s passes through $P$ twice\\ 
&$\iff$ $P$ is a self-intersection point of $\mathcal{F}$. 
\end{tabular} 
\end{thm}

\begin{figure}[h] 
\begin{tabular}{ccc}
\begin{minipage}[t]{0.3\linewidth}
\centering
\includegraphics[width = 45mm]{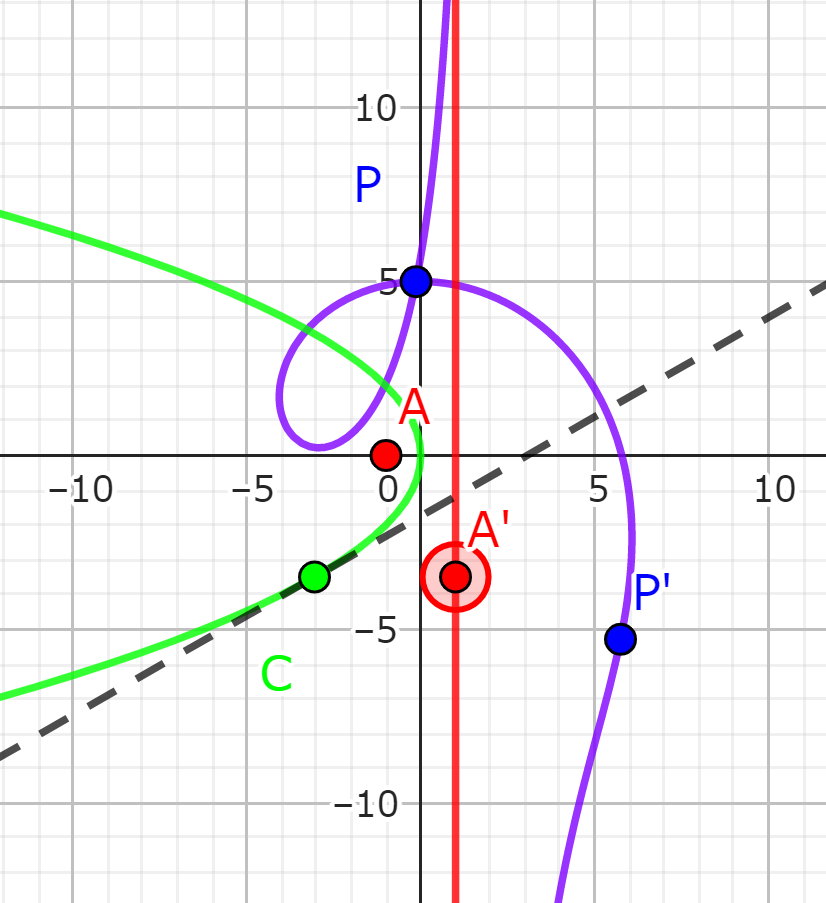}
\end{minipage}
&
\begin{minipage}[t]{0.3\linewidth}
\centering
\includegraphics[width = 45mm]{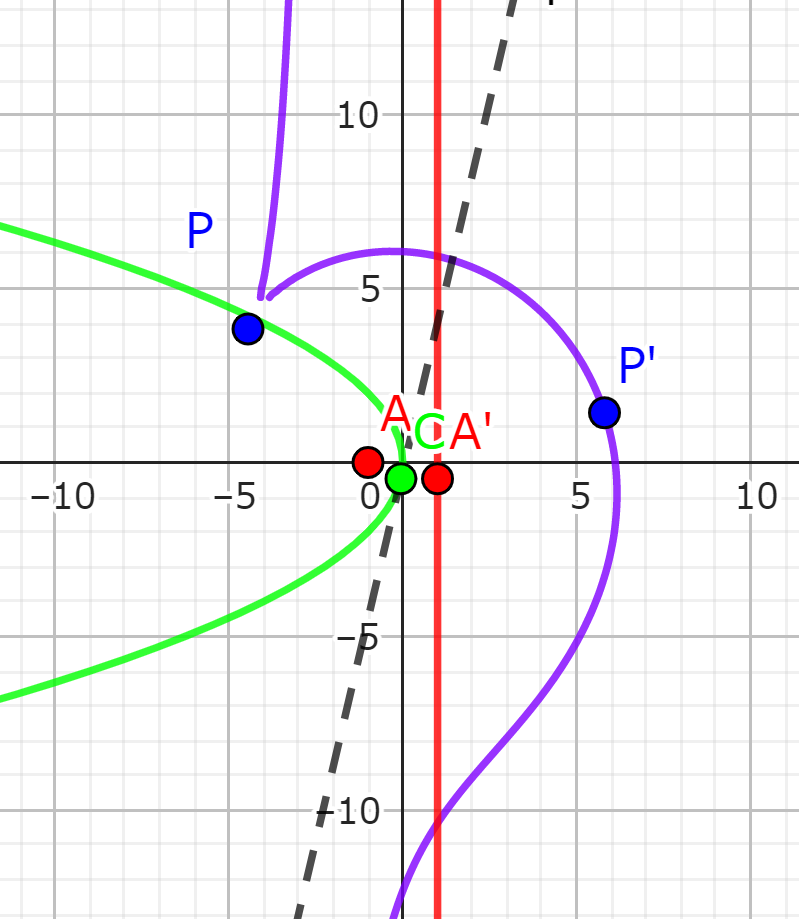}
\end{minipage}
&
\begin{minipage}[t]{0.3\linewidth}
\centering
\includegraphics[width = 45mm]{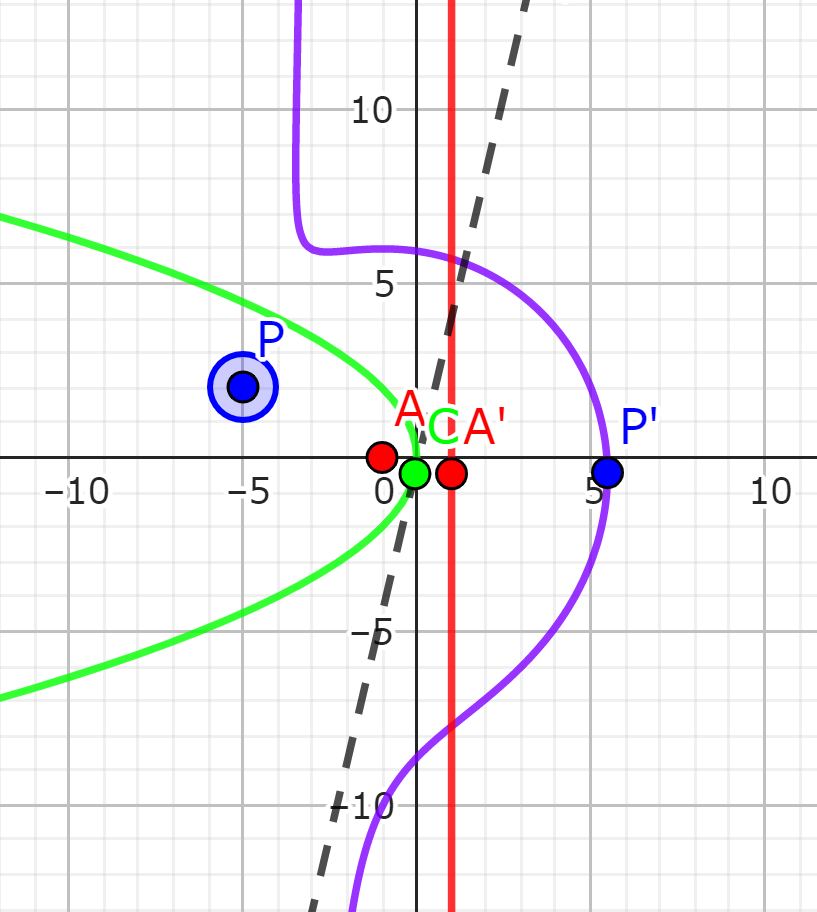}
\end{minipage}
\end{tabular}
\caption{The relationship between $P$ and $\mathcal{G}$ determines the shape of $\mathcal{F}$} 
\label{fig.PGF}
\end{figure}
In addition, we show that if $\mathcal{F}$ has an self-intersection at $P$, then the rotation number of the closed curve $\mathcal{F}_c$ contained in $\mathcal{F}$ around the point $A$ is determined by the relationship between $P$ and $x=1$ (\cref{thm3_9}). 
As a corollary of \cref{thm.main}, we also prove that the shapes of real cubic curves $a_0y^2-a_1xy^2-a_2xy-a_3x^2-a_4x^3=0$ may be classified similarly (\cref{thm4_1}). 

\cref{thm.main} asserts a relationship between Axioms 5 and 6.  
In general, given a point $A$, $Q$, and a line $m$, Axiom 5 allows the fold along a line $l$ such that $l$ goes through $Q$ and maps $A$ onto $m$. 
In the situation of \cref{thm.main}, each fold line $l$ is a tangent line of the parabola $\mathcal{G}$ at $Q(-2r,r^2)$ (\cref{lem.tangent}), so it is equivalent to consider all fold lines $l$ such that $l$ is a tangent line of $\mathcal{G}$ and the folding along $l$ maps $A$ onto the line $x=1$. 
Lang \cite[p.47]{LangRJ2010} considers all lines $l$ satisfying the condition (II) in \cref{thm.construction} and briefly studies the cubic curve that appears as the orbit of points that is the reflections of $A$ in $l$'s.
On the other hand, we cannot find studies that appear if we consider all lines satisfying the condition (I). 
Our study would shed new light on the relationship between Axioms 5 and 6, 
from the viewpoint of Beloch's curve (cf.~\cref{rem.Axioms56}).\\

This paper is organized as follows. 
In \cref{sec.kiso}, we prove \cref{thm.construction}. 

In \cref{sec.orbitofP'}, we prove \cref{thm2_1} asserting that 
the point $P$ is on the left side of (resp. on, or on the right side of) the parabola $\mathcal{G}:4x+y^2=0$ if and only if $\mathcal{F}$ is the disjoint union of the orbit of $P'$ and $\{P\}$ (resp. the orbit of $P'$). 

In \cref{sec.propertyofF}, we further study the curve $\mathcal{F}$ and the orbit of $P'$ and prove \cref{thm.main}. 
In \cref{ss.PandF}, we prove that $P$ is a uniquely existing singular point of $\mathcal{F}$ (\cref{pro3_1}). In addition, we prove that 
$P$ is on the left side of (resp. on, on the right side of) the parabola $\mathcal{G}:4x+y^2=0$ if and only if the orbit of $P'$ does not pass through $P$ (resp. passes through $P$ once, twice), 
by using the circle $\mathcal{C}$ with its center at $P$ and its radius as $AP$. 

In \cref{ss.APA'P'}, we obtain several conditions equivalent to that the segments $AP$ and $A'P'$ intersect (\cref{pro3_5}, \cref{lem3_6}), which play important roles in the proof of \cref{thm.main}. 

In \cref{ss.AandF}, we obtain the following result (\cref{thm3_9}) that is independent of \cref{thm.main}:
we have $p<1$ (resp. $p=1$, $p>1$) if and only if the rotation number of the orbit of $P'$ around $A$ is 0 (resp. undefined, 1). 

In \cref{ss.GandF}, we study the relation between the parabola $\mathcal{G}$ and the orbit of $P'$. 
We prove that the fold line is a tangent line of $\mathcal{G}$ (\cref{lem.tangent}), which is related to Axiom 5, 
and the conditions in \cref{thm3_4} are also equivalent to that the intersection $\mathcal{F}\cap \mathcal{G}$ consists of 0 (resp. 1,2) points \cref{thm3_11}. 

In \cref{ss.shapeofFatP}, we complete the proof of \cref{thm.main}. 
We may translate the relationship between $P$ and $\mathcal{G}$ into the condition of the sign of the Hessian $\mathcal{H}_\mathcal{F}=-4(4p+q^2)$ at $P$. 
As a consequence of Propositions \ref{pro3_5} and \cref{lem3_6}, we obtain that 
the orbit of $P'$ does not pass through $P$ (resp. passes through $P$ just once, more than once) if and only if $P$ is an isolated point (resp. a cusp, a self-intersection point) of $\mathcal{F}$ (\cref{thm3_12}). 
In \cref{ss.surface}, we visualize our situation by considering the surface $z=F(x,y)$ in $\mathbb{R}^3$. 

In \cref{sec.general}, we study a slightly general situation. We replace $A(-1, 0)$ and $P$ with $A(-\frac{\alpha}{2}, 0)$ and the origin $O$ respetively. By using \cref{thm3_12}, we prove that the shapes of cubic curves 
$a_0y^2-a_1xy^2-a_2xy-a_3x^2-a_4x^3=0$
are classified into three classes according to the relationship between the origin and a certain parabola (\cref{thm4_1}).% 

%\noindent \textbf{謝辞.}
\begin{ack}
The author would like to express her sincere gratitude to Kazuhiko Maruki and Isao Nakai for introducing such an excellent topic, 
to Kiyoshi Ohba for letting her know the importance of thinking in her own words in mathematics,
and Jun Ueki for his thorough guidance in making her ideas into an article. 
The author is also grateful to her parents, friends, and all others for their continuous support and help throughout this valuable experience. 
\end{ack}

%本資料を作成するにあたり, 議論に対するアドバイスや本資料の添削していただいた植木潤先生 (お茶の水女子大学, 講師)に深く感謝申し上げます.
%そして, 私の進捗状況を気にかけて下さった大場清先生 (同, 准教授), 
%折り紙について考えるきっかけを下さった中居功先生 (同, 名誉教授), %教授, 
%問いかけを下さった丸木和彦先生 (熊谷女子高等学校)%高校教諭
%をはじめとする多くの先生方, 数学科生の皆さまのおかげで研究をすることが出来ました.
%本当にありがとうございました.

\section{A proof of \cref{thm.construction}} 
\label{sec.kiso}

We prove \cref{thm.construction} for the convenience of readers. 

\begin{proof} 
First, we prove the assertion (1). 
Suppose that a fold $l$ satisfies the conditions (I) and (II) and let $(0,r)$ be the intersection of $l$ and the $y$-axis. Then $l$ is given by \[l:\ x+ry-r^2=0.\]
Let $P'(s,a-c)$ denote the reflection of $P$ in $l$. 
Since the midpoint $(\frac{b+s}{2}, a)$ of the segment $PP'$ is on the line $l$, we have 
\begin{equation} \label{eq1_01}
\frac{b+s}{2}+ra-r^2=0. \end{equation}
Since $l$ and $PP'$ are orthogonal, the vector 
$\ora{PP'}=\spmx{s-b\\(a-c)-(a+c)}=\spmx{s-b\\-2c}$ 
and the normal vector 
$\spmx{1\\r}$ of $l$ are parallel, and hence 
\begin{equation} \label{eq1_02}
\begin{vmatrix}s-b&1\\-2c&r\end{vmatrix}=r(s-b)+2c=0. \end{equation}
By eliminating $s$ in the simultaneous equations Eq.(\ref{eq1_01}) and Eq.(\ref{eq1_02}), we obtain 
\begin{equation} \label{eq1_1}
r^3-ar^2-br+c=0.
\end{equation}
Thus, $r$ is a solution of $x^3-ax^2-bx+c=0$. This completes the proof of (1). 

Next, we prove the assertion (2). 
Suppose that $r\in \mathbb{R}$ is any solution of $r^3-ar^2-br+c=0$ and 
consider the fold line $l$ which $A(-1,0)$ maps onto $A'(1,2r)$. 
Then, this $l$ satisfies the condition (I), and the $y$-intercept of $l$ turns out to be $r$. 
It suffices to show that the reflection $P'(s,t)$ of $P(b, a+c)$ in $l$ satisfies the condition (II), that is, $t=a-c$ holds. 

Since $l$ passes through $(0,r)$ and its normal vector is %$\bigl(\begin{smallmatrix}1\\r\end{smallmatrix}\bigl)$, 
$\spmx{1\\r}$, we have 
\[l:x+r(y-r)=0.\]
Since $l$ is parallel to %$\bigl(\begin{smallmatrix}r\\-1\end{smallmatrix}\bigl)$
$\spmx{r\\-1}$
 and since $PP'$ and $l$ are orthogonal, we obtain 
\begin{equation} \label{eq1_03}
\begin{pmatrix}s-b\\t-(a+c)\end{pmatrix}\cdot \begin{pmatrix}r\\-1\end{pmatrix}=r(s-b)-(t-(a+c))=0.
\end{equation}
In addition, since the midpoint $(\frac{b+s}{2}, \frac{a+c+t}{2})$ the segment $PP'$ is on $l$, we have 
\begin{equation} \label{eq1_04}
\frac{b+s}{2}+\frac{a+c+t}{2}r-r^2=0.
\end{equation}
By eliminating $s$ in the simultaneous equations Eq.(\ref{eq1_03}) and Eq.(\ref{eq1_04}), we obtain 
\begin{equation} \label{eq1_2}
-2r^3+(a+c+t)r^2+2br+(t-(a+c))=0.
\end{equation}
By Eq.(\ref{eq1_2})+Eq.(\ref{eq1_1})$\times 2$, we obtain 
\[(-a+c+t)r^2+(t-a+c)=(t-(a-c))(r^2+1)=0,\]
hence $t=a-c$. 
Thus, if the parameter $r$ continuously moves all $\mathbb{R}$, then Axiom 6 yields all solutions of $(\ast)$. This completes the assertion (2). 
\end{proof} 

\section{The orbit of $P'$ and Beloch's curve $\mathcal{F}$} \label{sec.orbitofP'} 
In order to precisely study the construction of \cref{thm.construction}, we consider all folds satisfying the condition (I) and study the orbit of $P'$'s that are the reflections of $P(b,a+c)$ in the folds. For simplicity, we set $(b,a+c)=(p,q)$. Then, we have the following. 

\begin{thm} \label{thm2_1} 
Consider the curve $\mathcal{F} : F(x,y)=2(q-y)^2-(q+y)(q-y)(p-x)-(p-x)^2(p+x)=0$ and 
the parabola $\mathcal{G}:4x+y^2=0$. 
The relationship between the orbit of $P'$ and $\mathcal{F}$ is as follows. 

$\bullet$ If $P$ is on the left side of $\mathcal{G}$, then $\mathcal{F}$ is the disjoint union of the orbit of $P'$ and $\{P\}$. 

$\bullet$ If $P$ is on $\mathcal{G}$ or on the right side of $\mathcal{G}$, then $\mathcal{F}$ coincides with the orbit of $P'$. 
\end{thm}
\begin{proof}
\underline{Proof of (the orbit of $P'$) $\subset$ $\mathcal{F}$}:  
Let $P'(s,t)$ be a point satisfying the condition. 
Since the corresponding fold $l$ is given by $x+ry-r^2=0$ and $PP'$ is perpendicular to $l$, we have 
\begin{equation} \label{eq.P'1}
\begin{pmatrix}s-p\\t-q\end{pmatrix}\cdot \begin{pmatrix}r\\-1\end{pmatrix}=r(s-p)-(t-q)=0. %\notag
\end{equation}
In addition, since the midpoint $(\frac{p+s}{2}, \frac{q+t}{2})$ of the segment $PP'$ is on $l$, we have 
\begin{equation} \label{eq.P'2}
\frac{p+s}{2}+\frac{q+t}{2}r-r^2=0. %\notag
\end{equation}
Now Eq.(\ref{eq.P'1}) yields $r(s-p)=t-q$, and Eq.(\ref{eq.P'2})$\times 2(s-p)^2$ yields 
\[F(s,t)=2(q-t)^2-(q+t)(q-t)(p-s)-(p-s)^2(p+s)=0.\] 

\underline{The condition for (the orbit of $P'$) $\supset$ $\mathcal{F}$}: 
If $(s, t)\in \mathcal{F}$, then we have 
\begin{equation} \label{eqtwothreeone}
F(s,t)=2(q-t)^2-((q+t)(q-t)-(p-s)(p+s))(p-s)=0.
\end{equation}
The perpendicular bisector $m$ of $P$ and $(s,t)$ is given by 
\[ \begin{pmatrix}x\\y\end{pmatrix}=\begin{pmatrix}q-t\\-p+s\end{pmatrix}k+\begin{pmatrix}\frac{p+s}{2}\\\frac{q+t}{2}\end{pmatrix}, \ k\in \mathbb{R},\]
\begin{equation} \label{eqtwothreetwo}
2(p-s)x+2(q-t)y-{(q+t)(q-t)+(p+s)(p-s)}=0.
\end{equation}
Note that $(s, t)\in$ (the orbit of $P'$) is equivalent to that the reflection $A'(u,2r)$ of $A(-1,0)$ in $m$ is on $x=1$. 
Since $m$ passes through the midpoint ($\frac{u-1}{2}, r$) of $AA'$, Eq.(\ref{eqtwothreetwo}) yields that 
\begin{equation}
\label{eq2_3}
(p-s)(u-1)+2r(q-t)-{(q+t)(q-t)+(p+s)(p-s)}=0.
\end{equation}
By Eq.(\ref{eq2_3})$\times (p-s)-$Eq.$(\ref{eqtwothreeone})$, we obtain
\begin{equation} \label{eq2_3_5}
(p-s)^2(u-1)+2(p-s)(q-t)r-2(q-t)^2=0.
\end{equation}
Since $m$ and $AA'$ are orthogonal, we obtain 
\begin{equation} \label{eq2_3_6}
(u+1)(q-t)=2r(p-s).
\end{equation}
By Eq.$(\ref{eq2_3_5})\cdot (u+1)^2$ and Eq.$(\ref{eq2_3_6})$, we obtain 
\begin{align*}
0&=(p-s)^2(u+1)^2(u-1)+4r^2(p-s)^2(u+1)-8r^2(p-s)^2\\
&=(p-s)^2((u+1)^2(u-1)+4r^2(u+1)-8r^2)\\
&=(p-s)^2((u+1)^2(u-1)+4r^2(u-1))\\
&=(p-s)^2(u-1)((u+1)^2+4r^2). 
\end{align*}
Thus we have $u=1$ or $s=p$ holds. 
If $u=1$, then $A'$ is on $x=1$. 
If $s=p$, then by $(s, t)\in \mathcal{F}$, we have $t=q$, $P=P'$, and hence $AP=A'P$. 
Let us consider when there exists $A'$ such that $AP=A'P$. 

Note that the parabola $\mathcal{G}$ is defined by the focus $A(-1,0)$ and the directrix $x=1$. 
Here, Fig.\ref{fig.PGF} in Section 1 helps. 
If $P$ is on the left side of $\mathcal{G}$, that is, if  $4p+q^2 <0$, then we always have $AP<A'P$, so no $A'$ satisfies the condition. 
%there does not exist $A'$ satisfying the condition. 
If otherwise, such $A'$ exists. 

Thus, we have verified that $\mathcal{F}\setminus$(the orbit of $P'$) has at most one element $P$, and we have $P\not\in$(the orbit of $P'$) if and only if $P$ is on the left side of $\mathcal{G}$. This completes the proof. 
\end{proof}

\begin{cor} \label{cor2_2}
The orbit of $P'(s,t)$ $\subset \mathcal{F}$ is parametrized by $r\in \mathbb{R}$ as follows: 
\[\pmx{s\\t}=\frac{1}{r^2+1}\pmx{(2+p)r^2-2qr-p\\ 2r^3-qr^2-2pr+q}.\]
\end{cor} 
\begin{proof}
In the former half of the proof of \cref{thm2_1}, if we take Eq.(\ref{eq.P'1})$-$Eq.(\ref{eq.P'2})$\times 2r$ and Eq.(\ref{eq.P'1})$\times r+$Eq.(\ref{eq.P'2})$\times 2$, then we obtain the assertion. 
\end{proof} 

\section{Properties of Beloch's curve $\mathcal{F}$} 
\label{sec.propertyofF}
In this section, we study properties of the real cubic curve $\mathcal{F}:F(x,y)=0$ 
that is the union of the orbit of $P'$ and $\{P\}$, 
and the orbit of $P'$ that is a curve parametrized by $r\in \mathbb{R}$ in $A'(1,2r)$.

\subsection{Relationships between $P$ and $\mathcal{F}$}
\label{ss.PandF}
\begin{prop} \label{pro3_1}
The point $P(p,q)$ is a uniquely existing singular point of $F(x,y)=0$. 
\end{prop} 
\begin{proof}
The partial derivatives of $F(x,y)$ are given by 
\begin{gather*}
F_x= (q+y)(q-y)+(p-x)(p+3x),\\
F_y= -4(q-y)+2y(p-x).
\end{gather*}
By solving $F=F_x=F_y=0$, we obtain $(x,y)=(p,q)$. 
We used Mathematica \cite{Mathematica} to calculate these values. 
\end{proof}

\begin{prop} \label{pro3_2} 
Suppose that $P$ is fixed. Then for a fold $l$ satisfying the condition {\rm (I)} in \cref{thm.construction} and points $A'$ and $P'$ determined by $l$, the following are equivalent. 
\begin{enumerate}
\item[{\rm (i)}] Points $P$ and $P'$ coincide. 
\item[{\rm (ii)}] The fold $l$ passes through $P$. 
\item[{\rm (iii)}] $AP=A'P$ holds. 
\item[{\rm (iv)}] The point $A'$ is on the circle $\mathcal{C}$ with its center at $P$ and its radius as $AP$. 
\end{enumerate}
\end{prop}

\begin{figure}[ht]
\centering
\includegraphics[width = 100mm]{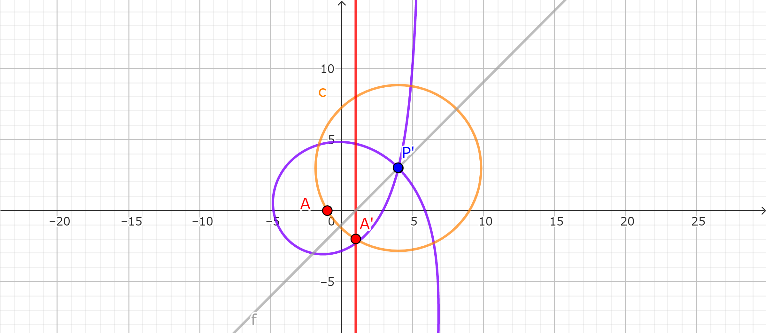}
\caption{A case satisfying the conditions (i) -- (iv) of \cref{pro3_2}}
\end{figure}

\begin{proof}
(i)$\iff$(ii): Since $P'$ is the reflection of $P$ in $l$, $P$ is on $l$ if and only if $P=P'$. 

(ii)$\iff$(iii): Since $l$ is the perpendicular bisector of $AA'$, $P$ is on $l$ if and only if $AP=A'P$. 

(iii)$\iff$(iv): This is obvious by the definition of the circle $\mathcal{C}$. 
\end{proof}

Among these, (i) $\iff$ (iv) enables us to study properties of $P'$ in the language of $A'$. 

\begin{prop} \label{pro3_3} 
The number of intersection points of $x=1$ and the circle $\mathcal{C}$ is determined by the relationship between $P(p,q)$ and the parabola $\mathcal{G}:4x+y^2=0$ as follows. 
\begin{itemize}
\item $x=1$ and $\mathcal{C}$ have no intersection point $\iff$ $4p+q^2<0$, 
\item $x=1$ and $\mathcal{C}$ have exactly one intersection point $\iff$ $4p+q^2=0$, 
\item $x=1$ and $\mathcal{C}$ have two intersection points $\iff$ $4p+q^2>0$. 
\end{itemize}
Note that $\mathcal{G}$ is determined by the focus $A(-1,0)$ and the directrix $x=1$. 
\end{prop}

\begin{proof} 
If $x=1$ and $\mathcal{C}$ has two intersection points, then the distance between the line $x=1$ and the center $P(p,q)$ is less than the radius $AP$, so we have $|AP|^2>(p-1)^2$ and hence $4p+q^2>0$. This implies that $P$ is on the right of $\mathcal{G}$. 
Other cases may be treated similarly. 
\end{proof}

The curve that is a connected component of $\mathcal{F}$ is the orbit of $P'$ under the condition (I) of \cref{thm.construction} and is parametrized $r\in \mathbb{R}$ in $A'(1,2r)$.

\begin{thm} \label{thm3_4} The orbit of $P'$ parametrized by $r\in \mathbb{R}$ admits the following equivalence. 
\begin{itemize}
\item The orbit of $P'$ does not pass through $P$ $\iff$ $4p+q^2<0$, 
\item The orbit of $P'$ passes through $P$ exactly once $\iff$ $4p+q^2=0$, 
\item The orbit of $P'$ passes through $P$ twice $\iff$ $4p+q^2=0$. 
\end{itemize}
\end{thm} 

\begin{proof} 
This is clear by Propositions \ref{pro3_2} and \ref{pro3_3}. 
\end{proof}

\subsection{The intersection of $AP$ and $A'P'$} %Conditions for $AP$ and $A'P'$ to intersect}
\label{ss.APA'P'} 

Here, we prove \cref{pro3_5} and \cref{lem3_6} that will play important roles in the proof of \cref{thm.main}. 

\begin{prop} \label{pro3_5} 
The relationship between $A$ and $\mathcal{C}$ determines how $AP$ and $A'P'$ intersect. 
%The way the segments $AP$ and $A'P'$ intersect determines the relationship between $A$ and $\mathcal{C}$. 
\begin{itemize}
\item $P=P'$ $\iff$ $A'$ is on $\mathcal{C}$. 
\item $AP\cap A'P'\neq \emptyset$ and $P\neq P'$ $\iff$ $A'$ is inside $\mathcal{C}$. 
\item $AP\cap A'P'= \emptyset$ $\iff$ $A'$ is outside $\mathcal{C}$. \end{itemize}
\end{prop}

\begin{proof} The case with $P=P'$ is already proved by \cref{pro3_2}. 
Suppose $P\neq P'$. 
Since $l$ is the common perpendicular bisector of $AA'$ and $PP'$, 
Each of $A$ and $A'$, $P$ and $P'$ are on the different side of $l$. 
We obtain the assertion by the following equivalences. 

\begin{tabular}{cl}
$AP\cap A'P'\neq \emptyset$ & $\iff$ $A$ and $P$ are on the different side of $l$\\
& $\iff$ The values of $x+ry-r^2$ at $A$ and $P$ have different signs\\
& $\iff$ $(-1-r^2)(p+qr-r^2) < 0$ \\
& $\iff$ $p+qr-r^2 >0$ \\ 
& $\iff$ $q - \sqrt{q^2+4p} < r < q + \sqrt{q^2+4p}$ \\
& $\iff$ $A'$ is inside of $\mathcal{C}$. \qedhere
\end{tabular}
\end{proof}

\begin{prop} \label{lem3_6} On the $x$-coordinates of $P(p, q)$ and $P'(s, t)$, we have 
\begin{itemize}
\item $P=P'$ $\iff$ $p=s$, 
\item $AP\cap A'P'\neq \emptyset$ and $P\neq P'$ $\iff$ $p>s$, 
\item $AP\cap A'P'= \emptyset$$\iff$ $p<s$. 
\end{itemize} 
\end{prop}

\begin{proof} The first case is obvious. Suppose $P\neq P'$. 
Since $l$ is the common perpendicular bisector of $AA'$ and $PP'$, 
$AA'$ and $PP'$ are parallel, and we have the following equivalence:

\begin{tabular}{cl}
$AP\cap A'P'\neq \emptyset $ &$\iff$ $\ora{AA'}=\spmx{2\\2r}$ and $\ora{PP'}=\spmx{s-p\\ t-q}$ are in the same direction \\
&$\iff$ $s>p$. 
\end{tabular}

\noindent 
This completes the proof. 
\end{proof}

\begin{figure}[h]
\begin{tabular}{cc}
\begin{minipage}[t]{0.5\linewidth}
\centering
\includegraphics[width = 40mm]{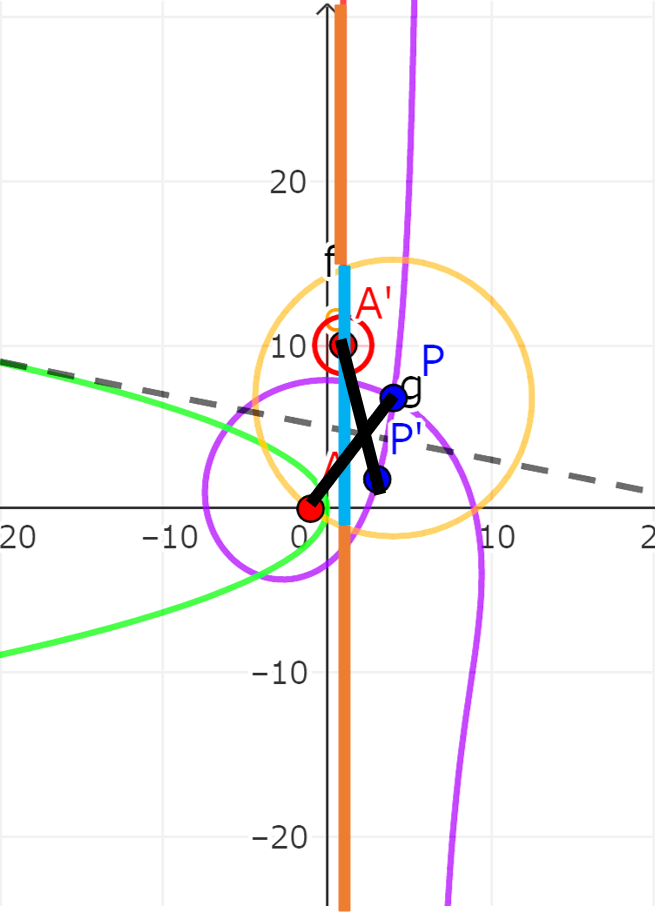}
\end{minipage}
&
\begin{minipage}[t]{0.5\linewidth}
\centering
\includegraphics[width = 40mm]{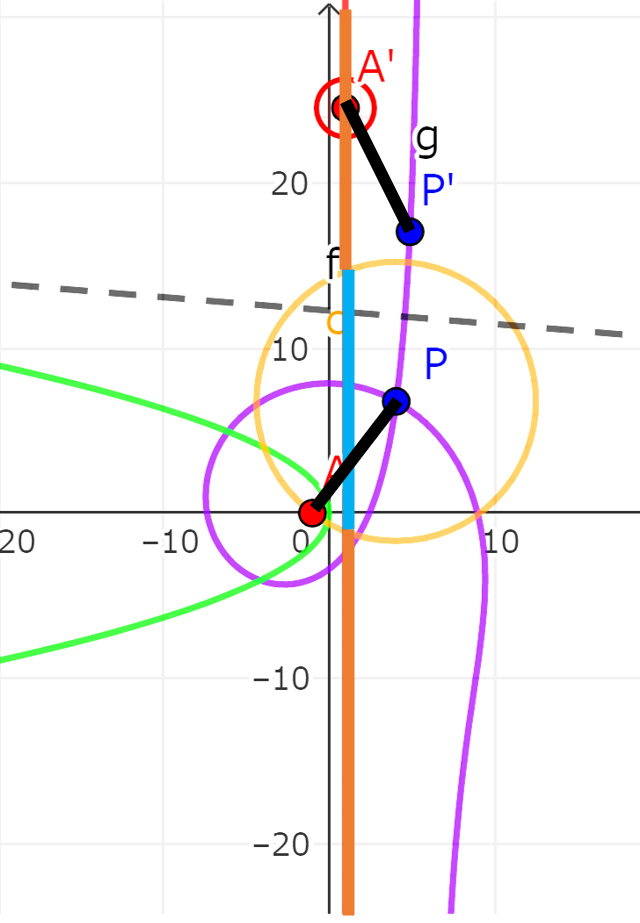}
\end{minipage}
\end{tabular}
\caption{The intersection of $AP$ and  $A'P'$} 
\end{figure}

\subsection{The rotation number of $\mathcal{F}_c$ around $A$}
\label{ss.AandF} 
Assume that $4p+q^2>0$, so that by \cref{thm3_4}, 
the curve $\mathcal{F}$ have a self-intersection at $P(p,q)$. 
Let $\mathcal{F}_c$ denote the closed curve between the self-intersections. 
In this subsection, we study the rotation number of $\mathcal{F}_c$ around $A(-1,0)$. 
Our result here is independent of \cref{thm.main}. 
\begin{dfn} \label{dfn3_8} 
Let $\Gamma$ be a piecewise smooth closed curve in $\mathbb{R}^2$ parametrized by $[0,1]\subset \mathbb{R}$ and let $A$ be a point that is not on $\Gamma$. 
If $P'$ goes along $\Gamma$ according to the parameter, then the rotation number of $\Gamma$ around $A$ stands for that of the image of $\varphi: \Gamma \rightarrow S^1;P' \mapsto \ora{AP'}/|\ora{AP'}|$. 
\end{dfn} 

\begin{lem} \label{lem.cross} 
Let $\Gamma$ be a piecewise smooth closed curve $\Gamma$ in $\mathbb{R}^2$ and 
let $A$ be a point that is not on $\Gamma$. 
If each of the half straight lines with their endpoint at $A$ and are parallel to $x$- or $y$-axis 
intersects with $\Gamma$ exactly once, then the rotation number of $\Gamma$ around $A$ is 1. 
\end{lem}

\begin{proof} 
Let $P_1, P_2, P_3, P_4$ be the intersections of $\Gamma$ and the half lines so that 
the $x$-coordinate of $P_1$ is greater than that of $A$ and $P_i$'s are numbered anti-clockwise. 
Assume that the suffixes are in $\mathbb{Z}/4\mathbb{Z}$. 
Take a small $\varepsilon >0$ and consider the circle $C_i$ with its center at $P_i$ and radius $\varepsilon$ for $i=1,2,3,4$. 
Suppose that $P'$ goes along $\Gamma$, let $P_i^-$ denote the last intersection of $P'$ and $C_i$ before reaching $P_i$, and let $P_i^+$ denote the first intersection of $P'$ and $C_i$ after starting $P_i$. 
By taking a sufficiently small $\varepsilon >0$, we may assume that 
each part $\Gamma'_i$ of $\Gamma$ connecting $P_i^{-}$ and $P_i^{+}$ does not have a self-intersection point.
By combining $\Gamma'_i$'s and the segments $L_i^{\pm}$'s such that each of them has an endpoint at $P_i^{\pm}$ and parallel to the $x$- or $y$-axises, we may obtain a cross-shaped simple close curve $\mathcal{X}$. 
In addition, for each $i$, let $\Gamma_i$ denote the closed part obtained by combining $L_{i+1}^-$, $L_i^+$, and the part of  $\Gamma$ connecting $P_i^+$ and $P_{i+1}^-$. 
Then, endowed with natural orientations, 
the curves satisfy $\Gamma=\mathcal{X}+\Gamma_1+\Gamma_2+\Gamma_3+\Gamma_4$ as 1-cycles. 
The rotation numbers of $\Gamma_i$'s around $A$ are zero and that of $\mathcal{X}$ is 1, so that of $\Gamma$ is 1. 
\end{proof} 

\begin{figure}[h]
\begin{center}
\includegraphics[width=35mm]{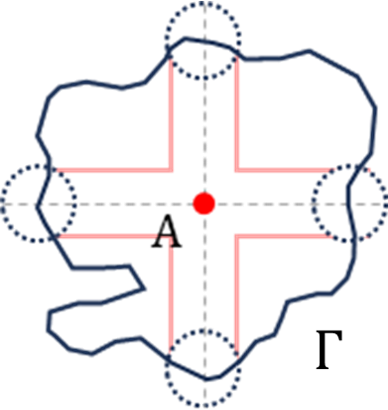}
\end{center} 
\caption{The cross-shaped closed curve in \cref{lem.cross}}
\end{figure}

\begin{lem} 
Suppose $4p+q^2>0$ and $p>1$. 
Then the defining polynomial $F(x,y)$ of the curve $\mathcal{F}$ satisfies the following. 

{\rm (1)} The equation $F(-1,t)=0$ in $t$ has two distinct solutions $\alpha,\beta$ whose signs differ. 

{\rm (2)} The equation $F(s,0)$ has three distinct solutions; One is on $s<-1$, and two are on $s>-1$. If we restrict the parameter $s$ to the part $\mathcal{F}_C$ of $\mathcal{F}$ between the self-intersection points, then $F(s,0)$ has two distinct solutions, one of which is in $s<-1$ and the other in $s>-1$. 
\end{lem}

\begin{proof}
(1) 
It suffices to show that $F(-1, t)=1 + p - p^2 - p^3 + q^2 - p q^2 - 4 q t + 3 t^2 + p t^2$
has $D>0$ and $\alpha\beta<0$.
We have \[\alpha\beta=(\text{constant term})/(\text{leading term})=(1 + p - p^2 - p^3 + q^2 - p q^2)/(3 + p)\]
and its sign coincides with that of $1 + p - p^2 - p^3 + q^2 - p q^2=(1-p)(1+2p+p^2+q^2)$. 
By $1-p<0$ and $1+2p+p^2+q^2>1-2p+p^2=(1-p)^2>0$, we obtain $\alpha \beta<0$. 
By \[D=(4q)^2-4(1 + p - p^2 - p^3 + q^2 - p q^2)(3+p)=(4q)^2-4\alpha\beta(3+p)^2>0,\]
we obtain the assertion. 

(2) 
$F(s, 0)=-p^3 + 2 q^2 - p q^2 + (p^2 + q^2) s + p s^2 - s^3$ is a cubic function in $s$ with its leading term being negative. 
By $-1<p$, we have 
\[F(-1, 0)=(1-p)(1+2p+p^2+q^2)<0, \ \ F(p, 0)=2q^2>0.\]  
Hence, there are one real solution in $s<-1$ and two real solutions in $s>-1$. 

Since $P(p,q)$ is the self-intersection point of $\mathcal{F}$, by $F(p,0)=2q^2>0$, 
if we restrict $s$ to the interval between the self-intersection points, 
then there exists exactly one real solution in each of the two regions. 
\end{proof}

\begin{thm}
\label{thm3_9} 
Suppose $4p+q^2>0$. Then, the relationship between the closed curve $\mathcal{F}_c$ and $A$ is determined by the place of $P$ as follows. 
\begin{itemize}
\item $p>1 \iff$ The rotation number of $\mathcal{F}_c$ around $A$ is 1, 
\item $p < 1 \iff$ The rotation number of $\mathcal{F}_c$ around $A$ is 0, or $A=P$, 
\item $p=1 \iff$ $A$ is on $\mathcal{F}_c$ and $A\neq P$. 
\end{itemize}
\end{thm}

\begin{proof} 
If $p>1$, then by the two lemmas above, the rotation number is 1. 

Suppose that $p<1$. If the rotation number is $\neq 0$, then 
the continuous map $\varphi:\mathcal{F}_c\to S^1; ;P'\mapsto \ora{AP'}/|\ora{AP'}|$ is surjective,
so there exists some $P'$ on $\mathcal{F}_c$ such that $\varphi(P')=\ora{PA}/|\ora{PA}|$, and the corresponding $A'$ is on the segment $AP$, contradicting to $p<1$. 
Thus, the rotation number is 0. 

The curve $\mathcal{F}$ passes through $A(-1,0)$ if and only if 
\[F(-1, 0)=-(p-1)((p+1)^2+q^2)=0, \]
that is, $p=1$ or $P(p,q)=A(-1,0)$ holds. 
If $p=1$, then $P\neq A$. This completes the proof. 
\end{proof} 

\begin{figure}[h]
\begin{tabular}{cccc}
\begin{minipage}[t]{0.2\linewidth}
\centering
\includegraphics[width = 35mm]{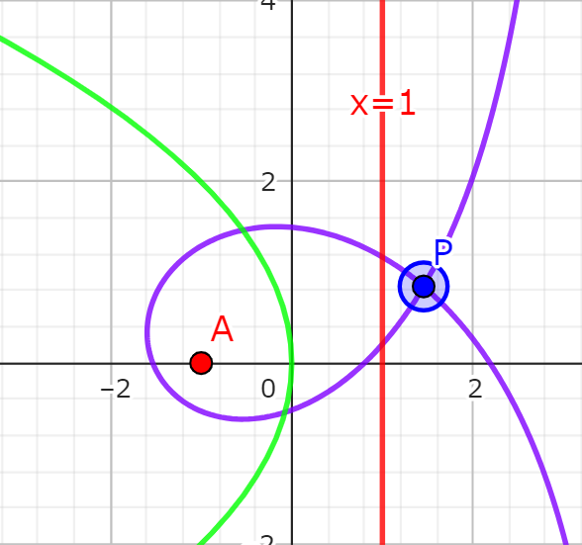}
\end{minipage} \ \ \ &

\begin{minipage}[t]{0.2\linewidth}
\centering
\includegraphics[width = 35mm]{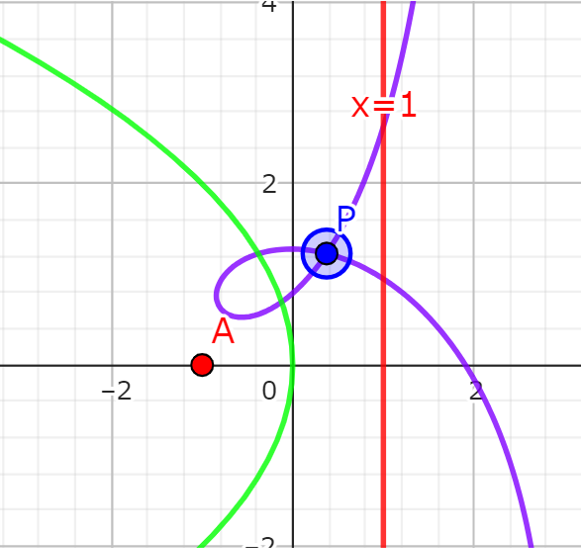}
\end{minipage} \ \ \ &

\begin{minipage}[t]{0.2\linewidth}
\centering
\includegraphics[width = 35mm]{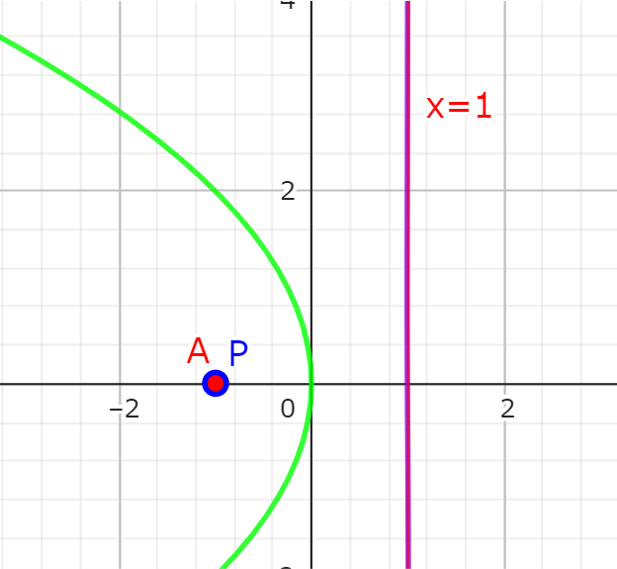}
\end{minipage} \ \ \ &

\begin{minipage}[t]{0.2\linewidth}
\centering
\includegraphics[width = 35mm]{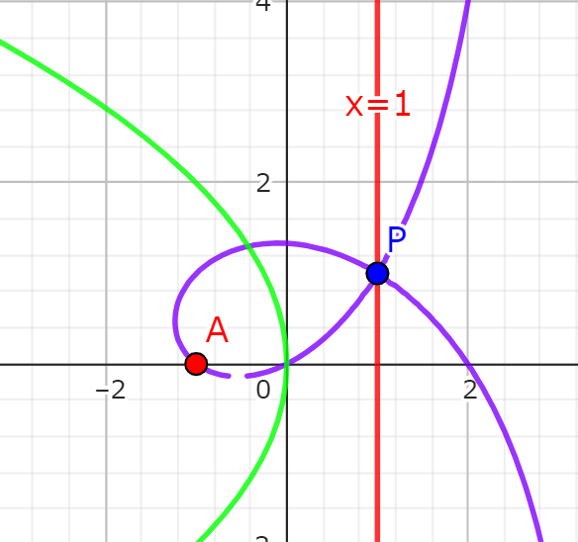}
\end{minipage}
\end{tabular}
\caption{The rotation number of $\mathcal{F}_c$ around $A$ and the place of $P$}
\end{figure}

\subsection{The intersection of $\mathcal{G} : 4x+y^2=0$ and $\mathcal{F}$}
%放物線$\mathcal{G} : 4x+y^2=0$と$\mathcal{F}$の関係} 
\label{ss.GandF}

\begin{lem}
\label{lem.tangent}
For any $r\in \mathbb{R}$, 
the perpendicular bisector of $A(-1, 0)$ and $A'(1, 2r)$ coincides with the tangent line 
of $\mathcal{G}:4x+y^2=0$ at $(-r^2, 2r)$. 
\end{lem}

\begin{proof} 
The line $l$ is given by $\ora{AA'}\cdot(\spmx{x\\y}-\dfrac{1}{2}(\ora{OA}+\ora{OA'}))=0$, that is, $x+ry-r^2=0$.
By the total differential $d(4x+y^2)=4dx+2ydy$, the tangent line of $\mathcal{G}$ at $(-r^2, 2r)$ is given by $4(x-(-r^2)+2(2r)(y-2r)=0$, that is, $x+ry-r^2=0$. Thus, these two coincide. 
\end{proof}

\begin{figure}[ht]
\centering
\includegraphics[width = 9.0cm]{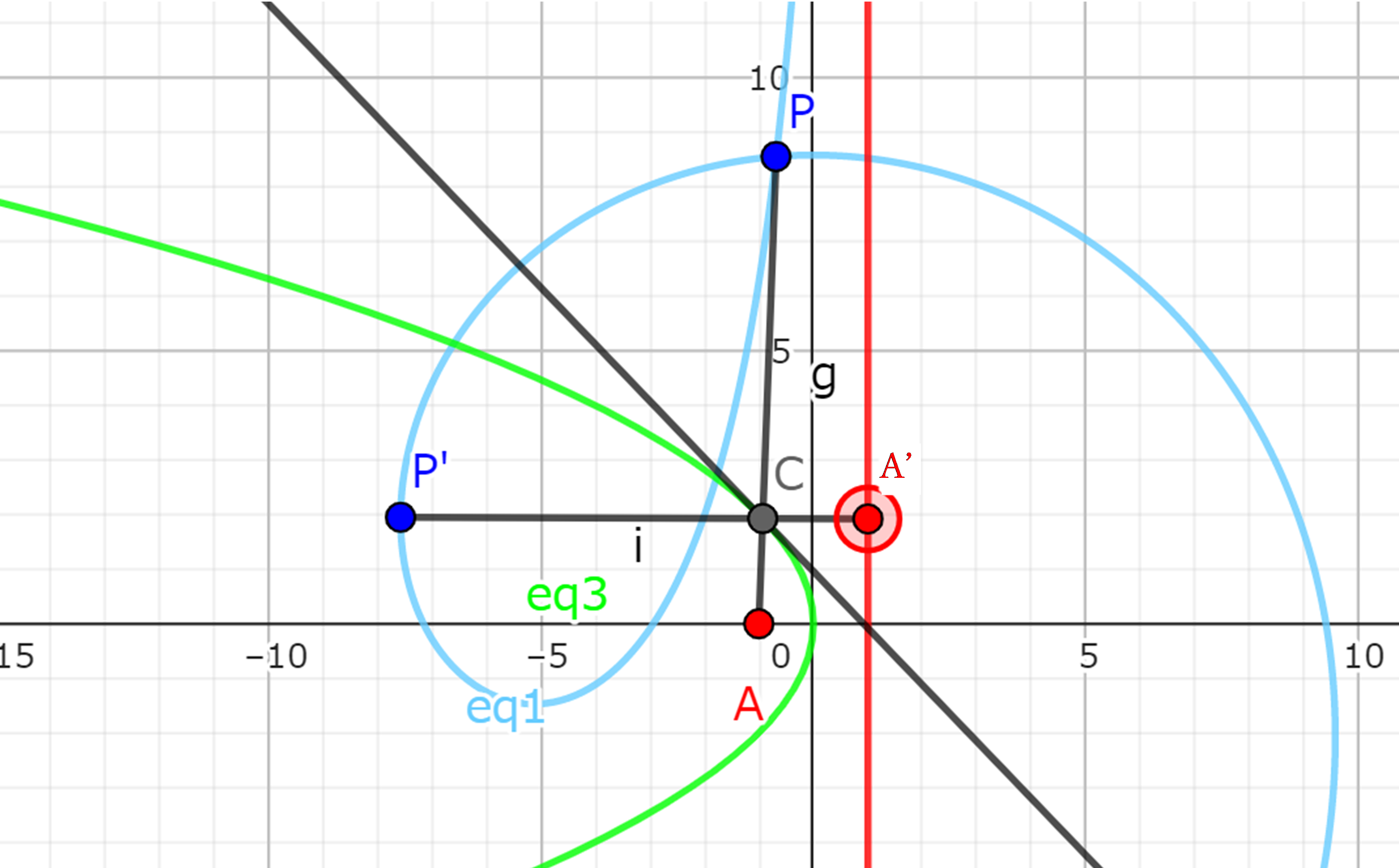}
\caption{The tangent line of $\mathcal{G}$ and $A'$} 
\end{figure}

\begin{thm} \label{thm3_11} 
The relationship between $P$ and $\mathcal{G}$ determines the number of intersection points of $\mathcal{F}$ and $\mathcal{G}$ as follows. 
\begin{itemize}
\item $4p+q^2<0$ $\iff$ $\mathcal{F}\cap \mathcal{G}=\emptyset$. 
\item $4p+q^2=0$ $\iff$  $\mathcal{F}\cap \mathcal{G}$ consists of one element. 
\item $4p+q^2>0$ $\iff$  $\mathcal{F}\cap \mathcal{G}$ consist of more than one element.
\end{itemize} 
\end{thm}

\begin{proof}
If $4p+q^2>0$, then $P(p,q)$ is on the right side of $\mathcal{G}$, while $A(-1, 0)$ is on the left side of $\mathcal{G}$, so the segment $AP$ and $\mathcal{G}$ intersect at some point $C(-c^2,2c)$. 
If $A'$ is at $(1,2c)$, then the points $P'$, $C$, $A$ are on a common line, and $A'P'$ is parallel to the $x$-axis. 
Note that $P$ is on the right side of $\mathcal{G}$ and the fold $l$ is the tangent line of $\mathcal{G}$ at $C$. 
Since $P', C, A'$ are placed on a line in this order, we see that $P'$ is on the left side of $\mathcal{G}$. 
Since the curve $\mathcal{F}$ restricted to $q-\sqrt{q^2+4p}\leq r\leq q+\sqrt{q^2+4p}$ is a closed curve, the intersection $\mathcal{F}\cap \mathcal{G}$ must have more than one element. 

If $4p+q^2=0$, then the intersection of the segment $AP$ and $\mathcal{G}$ is $C=P$. 
Hence if $A'$ is at $(1,2q)$, then $P=P'$ and $P'$ is on $\mathcal{G}$.  
If $A'\neq (1,2q)$, then $P'$ is on the right side of $\mathcal{G}$, since the fold line is the tangent line of $\mathcal{G}$. 
Thus, $\mathcal{F}\cap \mathcal{G}$ consists of just one element. 

If $4p+q^2<0$, then both $A$ and $P$ are on the left of $\mathcal{G}$, 
so the segment $AP$ and $\mathcal{G}$ does not intersect. 
In addition, since $P$ is on the left side of $\mathcal{G}$ and the fold line is the tangent line of $\mathcal{G}$, $P'$ must be on the right side of $\mathcal{G}$, 
yielding that $\mathcal{F}\cap \mathcal{G}=\emptyset$. 
\end{proof}

\begin{figure}[ht]
\begin{tabular}{cc}
\begin{minipage}[t]{0.34\linewidth}
\centering
\includegraphics[width = 4cm]{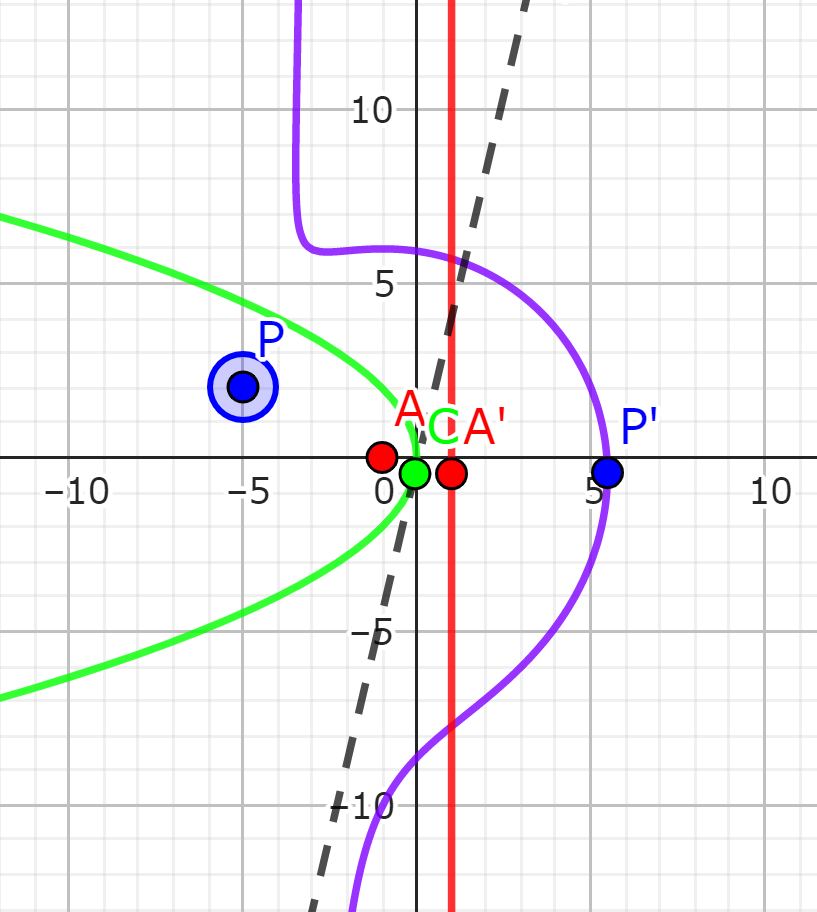}
\end{minipage}

\begin{minipage}[t]{0.34\linewidth}
\centering
\includegraphics[width = 4cm]{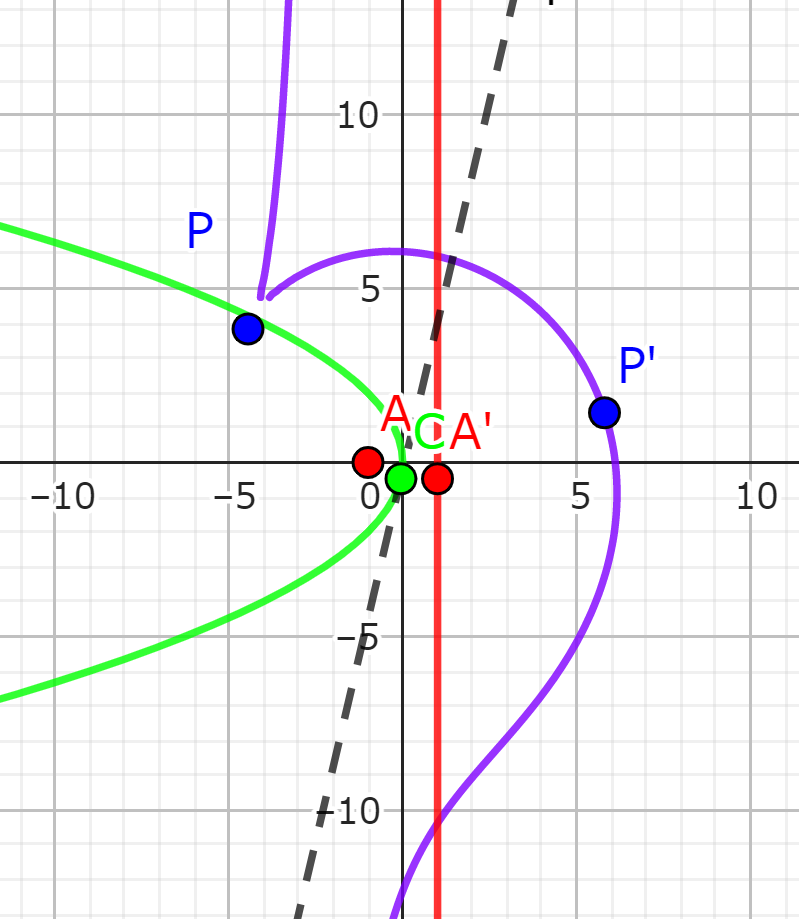}
\end{minipage}

\begin{minipage}[t]{0.34\linewidth}
\centering
\includegraphics[width = 4cm]{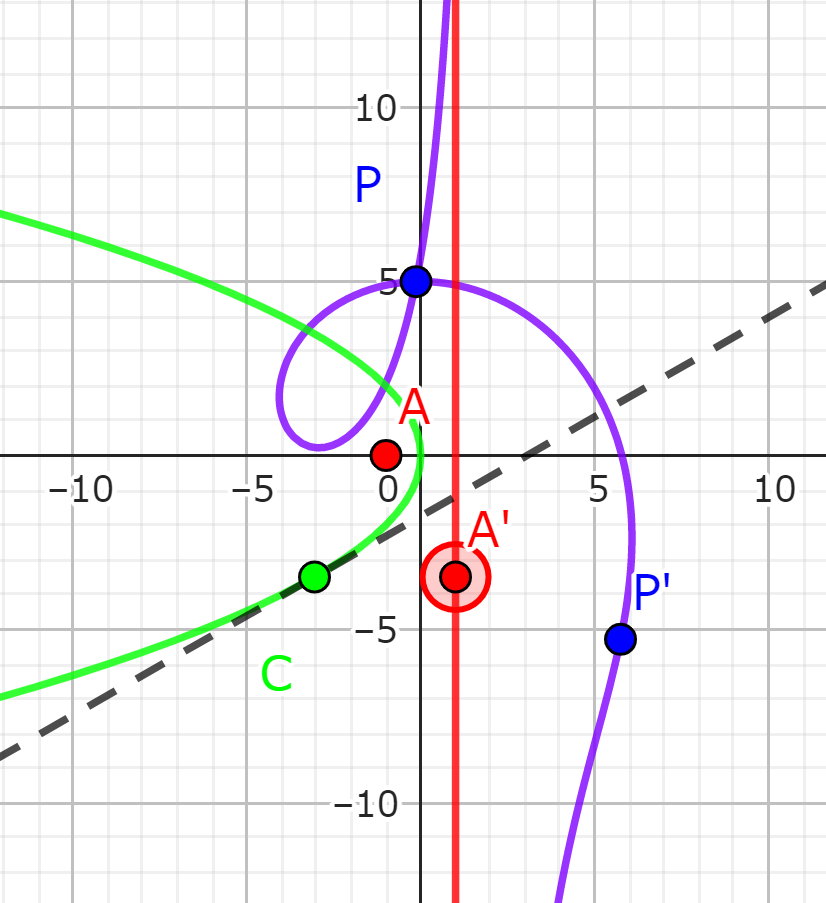}
\end{minipage}
\end{tabular}
\end{figure}

\begin{figure}[ht]
\begin{tabular}{cc}
\begin{minipage}[t]{0.34\linewidth}
\centering
\includegraphics[width = 4cm]{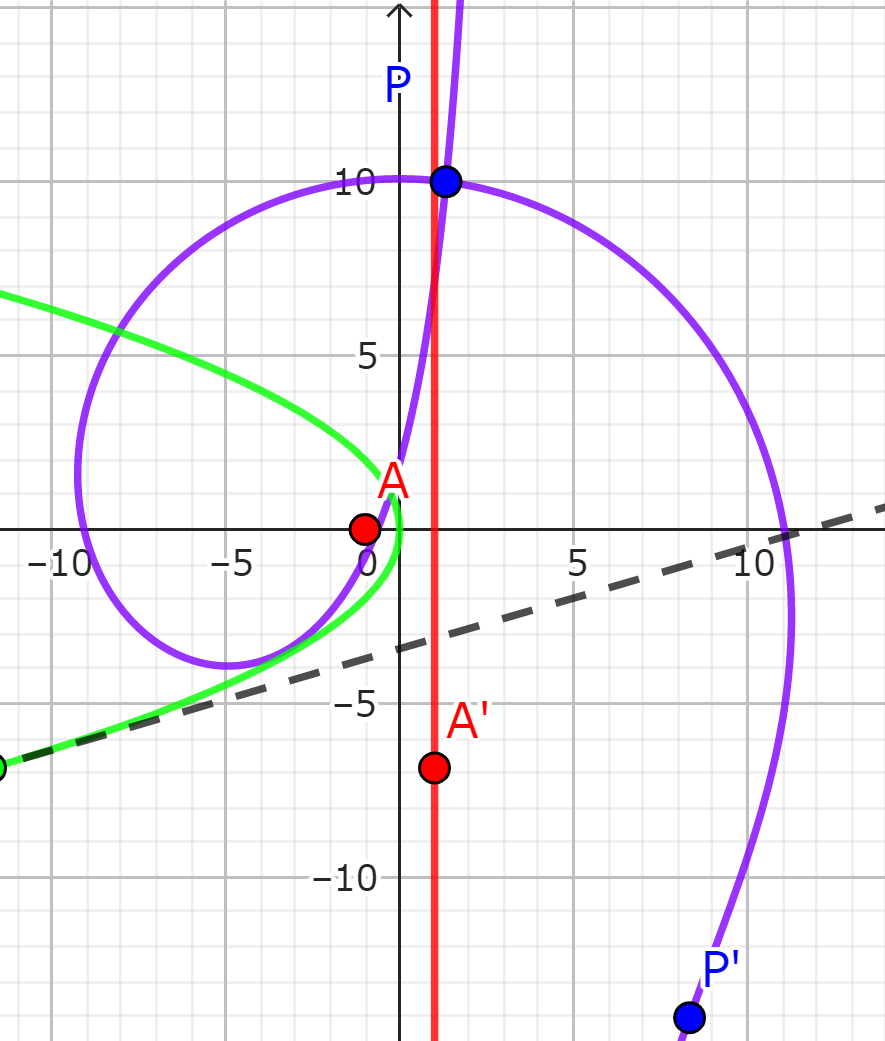}
\end{minipage}

\begin{minipage}[t]{0.34\linewidth}
\centering
\includegraphics[width = 4cm]{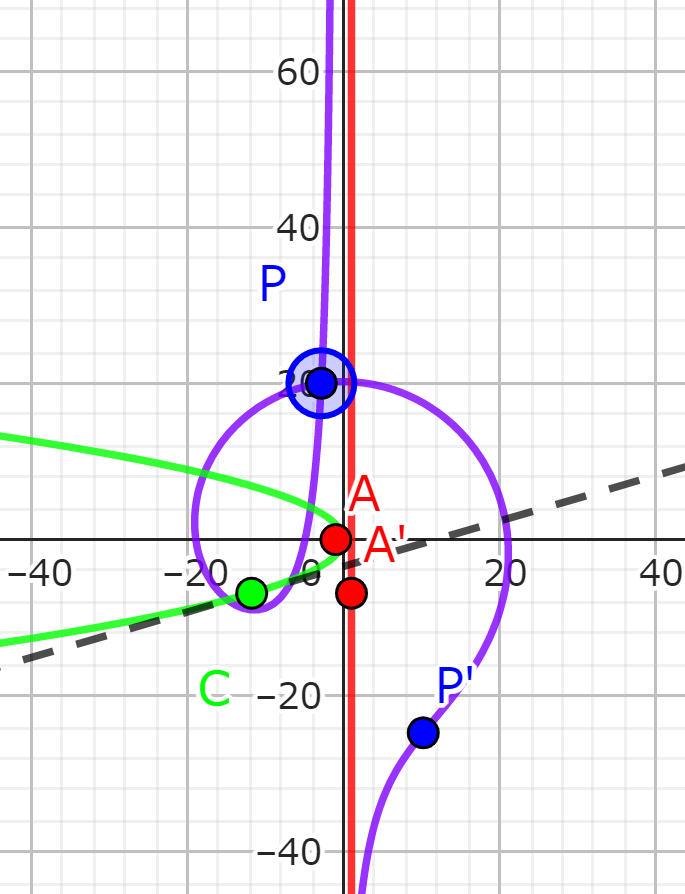}
\end{minipage}
\end{tabular}
\caption{The intersection of $\mathcal{F}$ and $\mathcal{G}$} 
\end{figure}

\subsection{The shape of $\mathcal{F}$ at $P$} 
\label{ss.shapeofFatP} 

\begin{thm} \label{thm3_12} 
The shape of $\mathcal{F}$ at $P$ is determined by the place of $P$ as follows. 
\begin{itemize}
\item $4p+q^2<0 \iff P$ is an isolated point of $\mathcal{F}$. 
\item $4p+q^2=0 \iff P$ is a cusp of $\mathcal{F}$. 
\item $4p+q^2>0 \iff P$ is a self-intersection point of $\mathcal{F}$. 
\end{itemize} 
Here, a cusp stands for a point such that $F_x=F_y=0$ holds there and it is neither an isolated point nor a self-intersection point. 
\end{thm} 
\begin{proof}
By \cref{pro3_1}, $P$ is a singular point of $F(x,y)=0$. 
The Hessian of $F(x,y)$ at $P$ is given by 
\begin{gather*}
\mathcal{H}_F=\begin{vmatrix}F_{xx}& F_{xy} \\ F_{yx}& F_{yy}\end{vmatrix}
=\begin{vmatrix}2p-6x& -2y \\ -2y & 4+2p-2x \end{vmatrix}
=(2p-6x)(4+2p-2x)-4y^2,\\ 
\mathcal{H}_{\mathcal{F}}(p, q)=-4(4p+q^2).
\end{gather*}
By elementary calculus for two variable functions, $z=F(x,y)$ takes local maximum/minimum value at $P(p,q)$ if $\mathcal{H}_{\mathcal{F}}(p, q)>0$, while it takes saddle point at $P(p,q)$ if $\mathcal{H}_{\mathcal{F}}(p, q)<0$. 
Since $\mathcal{F}$ is the intersection of $z=F(x,y)$ and $z=0$, 
the condition $4p+q^2<0$ implies that $P$ is an isolated point, 
while $4p+q^2>0$ implies that $P$ is a self-intersection point. 
Suppose $4p+q^2=0$ and note that $\mathcal{C}$ and $x=0$ have only one intersecting point. Then, \cref{pro3_5} yields that $AP$ and $A'P'$ must intersect at $P$, 
and \cref{lem3_6} yields that $P'(s,t)$ satisfies $p\leq s$. 
By that facts that the orbit of $P'$ is connected, $p\leq s$, and $P$ is a singular point of $\mathcal{F}$ (\cref{pro3_1}), we see that $P$ is a cusp of $\mathcal{F}$. 
\end{proof}

\begin{proof}[Proof of {\rm \textbf{\cref{thm.main}.}}] 
All above complete the proof of \cref{thm.main}. 
\end{proof}

\begin{rem} \label{rem.Axioms56}
As we mentioned in the Introduction, 
Axion 5 asserts that if points $A,Q$ and a line $m$ are given, then we may find a fold $l$ that passes through $Q$ and that maps $A$ onto $m$, 
and it is equivalent to that we may construct a tangent line of a given parabola satisfying a certain condition. 
By \cref{lem.tangent}, we can regard \cref{thm.main} as that we consider all tangent lines of the parabola $\mathcal{G}$ that passes through $Q$ and that maps $A$ onto $x=1$. 

On the other hand, Axiom 6 is equivalent to that we may construct common tangent lines of given two parabolas. 
Note that $\mathcal{G}:4x+y^2=0$ is determined by its focus $A(-1,0)$ and the directrix $x=1$.
Let $\mathcal{G'}$ denote the parabola determined by its focus $P(p,q)$ and the directrix $y=a-c$. 
Then the fold line $l$ is a common tangent line of $\mathcal{G}$ and $\mathcal{G'}$. 

Our study in this paper raises a viewpoint to compare these two Axioms. 
%It would be interesting instead to focus on 
If we take $\mathcal{G}'$ into account as well, then we may find a further perspective. 
\end{rem} 

\begin{rem} 
We considered all folds satisfying the condition (I) in \cref{thm.construction} and determined Beloch's curve $\mathcal{F}:F(x,y)=0$ by the orbit of $P'$ that is the reflection of $P$. 
If we instead consider all folds satisfying (II), then the orbit of $A'$ which is the reflection of $A$ 
also defines a real cubic curve and admits a similar study (cf.~\cite[p.47]{LangRJ2010}). 
\end{rem}

\subsection{The surface $z=F(x,y)$} \label{ss.surface}
If we regard $\mathcal{F}$ as the intersection of $z=F(x,y)=2(q-y)^2-(q+y)(q-y)(p-x)-(p-x)^2(p+x)$ and $z=0$, then we may capture the situation in a 3-dimensional view (Fig.\ref{fig.3D}).

\begin{figure}[ht] 
\begin{tabular}{ccc}
\begin{minipage}[t]{0.34\linewidth}
\centering
\includegraphics[width = 60mm]{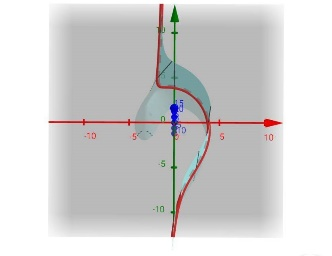}
\end{minipage} \hspace{-10mm}
&
\begin{minipage}[t]{0.34\linewidth}
\centering
\includegraphics[width = 60mm]{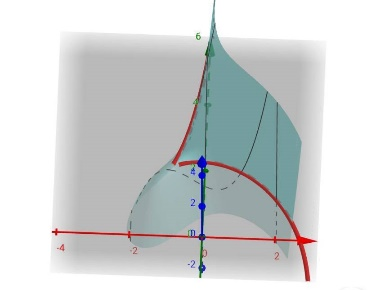}
\end{minipage} \hspace{-10mm}
&
\begin{minipage}[t]{0.34\linewidth}
\centering
\includegraphics[width = 60mm]{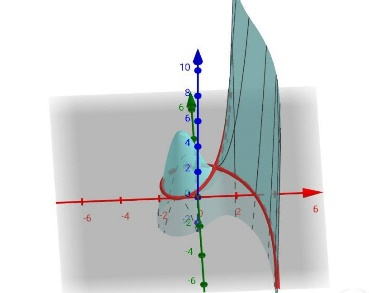}
\end{minipage}
\end{tabular} 
\caption{The surface $z=F(x,y)$ and the cure $\mathcal{F}$}\label{fig.3D}
\end{figure}

If $\mathcal{H}_F(p,q)=-4(4p+q^2)>0$, then by $F_{xx}(p,q)=-4p>0$, we see that $z=F(x,y)$ takes a local maximum value at $P$. 
We may expect the following. 

\begin{conj} 
The shape of $z=F(x,y)$ is determined by the place of $P(p,q)$ as follows. 

\begin{tabular}{cl}
$\bullet$ $4p+q^2<0 \iff$ & $z=F(x,y)$ has a unique local maximum pint at $(p,q,0)$\\
&and a unique saddle point.\\
$\bullet$ $4p+q^2=0 \iff$ & $z=F(x,y)$ has no local maximum/minimum point \\
&and a unique saddle point at $(p,q,0)$. \\
$\bullet$ $4p+q^2>0 \iff$ & $z=F(x,y)$ has a unique local maximum point \\
& and a unique saddle point $(p,q,0)$.
\end{tabular} 
\end{conj} 
%\begin{itemize}
%\item $4p+q^2<0 \iff$ $z=F(x,y)$ has a unique local maximum pint at $(p,q,0)$ and a unique saddle point.
%曲面は1つの極大点と1つの峠点を持ち, 極大点は$(p,q,0)$. 
%\item $4p+q^2=0 \iff$ $z=F(x,y)$ has no local maximum/minimum point and has a unique saddle point at $(p,q,0)$. 
%曲面は極値点を持たず, 1つの峠点$(p,q,0)$を持つ. 
%\item $4p+q^2>0 \iff$ $z=F(x,y)$ has one local maximum point and one saddle point $(p,q,0)$. 
%曲面は1つの極大点と1つの峠点を持ち, 峠点は$(p,q,0)$. 
%\end{itemize} 

\section{Real cubic curve $a_0y^2-a_1xy^2-a_2xy-a_3x^2-a_4x^3=0$}
%実代数曲線 $a_0y^2-a_1 xy^2-a_3x^2-a_4x^3=0$の概形} 
\label{sec.general}
Let $\alpha\in \mathbb{R}$ and let $P(p,q)$ be any point. If we consider all folds such that $A_{P}(-\frac{\alpha}{2}, 0)$ map onto the line $x=\frac{\alpha}{2}$, then the union of the orbit of the reflections $P'$ of $P$ about the folds and the set $\{P\}$ of a possibly isolated point is given by 
\[ \mathcal{F}_{P}:\alpha(q-y)^2-(q-y)(q+y)(p-x)-(p-x)^2(p+x)=0. \]
We have considered the case with $\alpha=2$ in the preceding section. 
Since the situations are in a scaling relation, the previous argument shows that 
the shape of this curve is determined by the relationship between the parabola $\mathcal{G}_{P}:4x+y^2=0$ and $P$,
and that $P$ is the uniquely existing singular point of this curve $\mathcal{F}_{P}$. 

By the parallel movement by $\ora{OP}$, we see that 
if we consider all folds that map $A_{O}(-\frac{\alpha}{2}-p, -q)$ onto the line $x=\frac{\alpha}{2}-p$, then the union of the orbit of the reflections of $O$ about the folds and the set $\{O\}$ of a possibly isolated point 
is given by 
\[ \mathcal{F}_{O}:\alpha y^2-xy^2-2qxy-2px^2-x^3=0. \]
The shape of this curve $\mathcal{F}_{O}$ is determined by the relationship between the parallel $\mathcal{G}_{O}:4(x+p)+(y+q)^2=0$ and $O$, and $O$ is a uniquely existing singular point of $\mathcal{F}_{O}$. 
Based on these observations, we may obtain the following. 

\begin{thm} \label{thm4_1} 
The real cubic curve 
\[ \widetilde{\mathcal{F}}:\ a_0y^2-a_1xy^2-a_2xy-a_3x^2-a_4x^3=0\ \ (a_i \in \mathbb{R}-\{0\}).\] 
has a unique singular point at $O$ and its shape is determined by the relationship between $O$ and the value of $G_O(0, 0)=\frac{2a_3}{\sqrt{a_1a_4}}+(\frac{a_2}{2a_1})^2$ as follows. 
\begin{itemize}
\item $G_O(0, 0)<0\iff$ $O$ is an isolated point of $\widetilde{\mathcal{F}}$. 
\item $G_O(0, 0)=0 \iff$ $O$ is a cusp of $\widetilde{\mathcal{F}}$. 
\item $G_O(0, 0)>0 \iff$ $O$ is a self-intersection point of  $\widetilde{\mathcal{F}}$. 
\end{itemize} 
\end{thm} 

\begin{proof}
If we multiply the formula of $\widetilde{\mathcal{F}}$ by $\beta=\sqrt{\frac{a_4}{{a_{1}}^3}}$, 
then we obtain 
\begin{gather*}
a_0\beta y^2-a_1\beta xy^2-a_2\beta xy-a_3\beta x^2-a_4\beta x^3=0, \\
a_0\beta y^2-(a_1\beta x)y^2-\frac{a_2}{a_1}(a_1\beta x)y-\frac{a_3}{{a_1}^2\beta }(a_1\beta x)^2-(a_1\beta x)^3=0.
\end{gather*}
Comparing this with $\mathcal{F}_{O}:\alpha y^2-xy^2-2qxy-2px^2-x^3=0$, 
we find 
\[\alpha=a_0\beta=a_0\sqrt{\frac{a_4}{{a_1}^3}}, \ p=\frac{a_3}{2{a_1}^2\beta}, \ q=\frac{a_2}{2a_1}.\]
Hence the preceding argument completes the proof. 
\end{proof}

\begin{eg} (1) If we put $(a_0,a_1,a_2,a_3,a_4)=(-b,-1,-a,0,-1)$, 
then $\widetilde{\mathcal{F}}$ becomes the standard form  $x(x^2+y^2)+(ax-by)y=0$ of so-called \emph{the ophiuride}. %\footnote{オニヒトデ} の標準形 $x(x^2+y^2)+(ax-by)y=0$ となる. 

(2) If we put $(a_0,a_1,a_2,a_3,a_4,a_5)=(2a,-1,0,0,-1)$, 
then $\widetilde{\mathcal{F}}$ becomes the standard form $x(x^2+y^2)-2ay^2=0$ of so-called \emph{the cissoid of Diocles}.

%（オニヒトデとアイビーの挿絵を入れる？ ）
%曲線$\widetilde{\mathcal{F}}$は
%the cissoid of Diocles の標準形 $x(x^2+y^2)-2ay^2=0$ となる. 
%\footnote{植物のIvyをギリシャ語でKissosという. その葉の形に似ていることからギリシャの幾何学者ディオクレスはkissoidと名付けた. そこから転じて英語圏ではcissoidと呼ばれており, 音訳から日本語では疾走線とも呼ばれる.} 
\end{eg} 

\bibliographystyle{alpha}
\bibliography{origami202312arXiv.bib}

\begin{thebibliography}{Lan10}

\bibitem[Bel36]{Beloch1936}
Margherita~Pianolla Beloch.
\newblock Sul metodo del ripiegamento della carta per la risoluzione dei
  problemi geometrici.
\newblock {\em Periodico di Mat. IV, XVI}, (2):104--108, 1936.

\bibitem[Hul11]{Hull2011AMM}
Thomas~C. Hull.
\newblock Solving cubics with creases: the work of {Beloch} and {Lill}.
\newblock {\em Am. Math. Mon.}, 118(4):307--315, 2011.

\bibitem[Inc23]{Mathematica}
Wolfram~Research{,} Inc.
\newblock Mathematica, {V}ersion 13.3, 2023.
\newblock Champaign, IL, 2023.

\bibitem[Jus86]{Justin1986}
Jacques Justin.
\newblock R\'{e}solution par le pliage de l'\'{e}quation du troisi\`{e}me
  degr\'{e} et applications g\'{e}om\'{e}triques.
\newblock {\em L'{O}uvert - {J}ournal de l'{APMEP} d'Alsace et de l'{IREM} de
  {S}trasbourg}, 42:9--19, 1986.
\newblock
  \url{https://mathinfo.unistra.fr/websites/math-info/irem/Publications/L_Ouvert/n042/o_42_9-19.pdf}.

\bibitem[Kat99]{Kato1999}
Koichi Kato.
\newblock Origami and formulas.
\newblock \url{http://izumi-math.jp/K_Katou/ori_h/ori_h.htm}, 1999.

\bibitem[Lan10]{LangRJ2010}
Robert~J. Lang.
\newblock Origami and geometric constructions.
\newblock
  \url{https://langorigami.com/wp-content/uploads/2015/09/origami_constructions.pdf},
  2010.

\bibitem[Lil67]{Lill1867}
E.~Lill.
\newblock R\'{e}solution graphique des \'{e}quations num\'{e}riques de tous les
  degr\'{e}s \`{a} une seule inconnue, et description d'un instrument
  invent\'{e} dans ce but.
\newblock {\em Nouvelles annales de math\'{e}matiques : journal des candidats
  aux \'{e}coles polytechnique et normale}, 6:359--362, 1867.

\bibitem[NO15]{NakaiOno2015}
Isao Nakai and Eri Ono.
\newblock Finding roots by a simple origami.
\newblock {\em Nat. Sci. Rep. Ochanomizu Univ.}, 66(1):1--17, 2015.

\end{thebibliography}

\ \\

\end{document}